\documentclass[12pt]{article}
\usepackage{amssymb, amsmath, amsthm, amsfonts,xcolor, enumerate}

\usepackage[shortlabels]{enumitem}
\usepackage{caption}
\usepackage{subcaption}
\usepackage{graphicx}
\usepackage{amssymb, amsmath, amsfonts,xcolor, enumerate}
\usepackage{fullpage}
\usepackage{hyperref}
\newtheorem{theorem}{Theorem}[section]
\newtheorem{lemma}[theorem]{Lemma}

\newtheorem{claim}[theorem]{Claim}

\theoremstyle{definition}
\newtheorem{rmk}[theorem]{Remark}
\theoremstyle{definition}

\theoremstyle{definition}

\theoremstyle{definition}
\newtheorem{defn}[theorem]{Definition}
\theoremstyle{definition}

\theoremstyle{definition}

\theoremstyle{definition}

\newcommand{\ep}{\varepsilon}

\newcommand{\de}{\delta}

\newcommand{\ga}{\gamma}
\newcommand{\De}{\Delta}

\newcommand{\cH}{\mathcal{H}}

\newcommand{\cG}{\mathcal{G}}

\newcommand{\cF}{\mathcal{F}}

\newcommand{\cB}{\mathcal{B}}

\newcommand{\MIS}{\mathrm{MIS}}
\newcommand{\MTF}{\mathcal{M}_3}
\newcommand{\cM}{\mathcal{M}}
\newcommand{\setm}{-}

\title{The typical structure of maximal triangle-free graphs}

\author{J\'{o}zsef Balogh, Hong Liu, \v{S}\'{a}rka Pet\v{r}\'{i}\v{c}kov\'{a}, Maryam Sharifzadeh}

\author{
 J\'ozsef Balogh
 \thanks{Department of Mathematical Sciences,
 University of Illinois at Urbana-Champaign, Urbana, Illinois 61801, USA {\tt
jobal@math.uiuc.edu}. Research is partially supported by Simons Fellowship, NSF
CAREER Grant
 DMS-0745185, Marie Curie FP7-PEOPLE-2012-IIF 327763 and Arnold O. Beckmann Research Award RB15006.}
 \quad
 Hong Liu
 \thanks{Department of Mathematical Sciences,
 University of Illinois at Urbana-Champaign, Urbana, Illinois 61801, USA {\tt
hliu36@illinois.edu}.}
\quad
\v{S}\'{a}rka Pet\v{r}\'{i}\v{c}kov\'{a}
 \thanks{Department of Mathematical Sciences,
 University of Illinois at Urbana-Champaign, Urbana, Illinois 61801, USA {\tt
petrckv2@illinois.edu}.}
 \quad
 Maryam Sharifzadeh
 \thanks{Department of Mathematical Sciences,
 University of Illinois at Urbana-Champaign, Urbana, Illinois 61801, USA {\tt
sharifz2@illinois.edu}.}
}

\begin{document}
\maketitle

\begin{abstract}
Recently, settling a question of Erd\H{o}s, Balogh and Pet\v{r}\'{i}\v{c}kov\'{a} showed that there are at most $2^{n^2/8+o(n^2)}$ $n$-vertex maximal triangle-free graphs, matching the previously known lower bound. Here we characterize the typical structure of maximal triangle-free graphs. We show that almost every maximal triangle-free graph $G$ admits a vertex partition $X\cup Y$ such that $G[X]$ is a perfect matching and $Y$ is an independent set.  

Our proof uses the Ruzsa-Szemer\'{e}di removal lemma, the Erd\H{o}s-Simonovits stability theorem, and recent results of Balogh-Morris-Samotij and Saxton-Thomason on characterization of the structure of independent sets in hypergraphs. The proof also relies on a new bound on the number of maximal independent sets in triangle-free graphs with many vertex-disjoint $P_3$'s, which is of independent interest.  
\end{abstract}


\section{Introduction}
Given a family of combinatorial objects with certain properties, a fundamental problem in extremal combinatorics is to describe the \emph{typical} structure of these objects. This was initiated in a seminal work of Erd\H{o}s, Kleitman, and Rothschild~\cite{Erdos} in 1976. They proved that almost all triangle-free graphs on $n$ vertices are bipartite, that is, the proportion of $n$-vertex triangle-free graphs that are not bipartite goes to zero as $n \to \infty$. Since then, various extensions of this theorem have been established.
The typical structure of $H$-free graphs has been studied when $H$ is a large clique~\cite{B6,KPR},  $H$ is a fixed color-critical subgraph~\cite{PS}, $H$ is a finite family of subgraphs~\cite{BBS2}, and $H$ is an induced subgraph~\cite{BB}. For sparse $H$-free graphs, analogous problems  were examined in~\cite{Balogh2, OPT}. 
In the context of other combinatorial objects, the typical structure of hypergraphs with a fixed forbidden subgraph is investigated for example in~\cite{BMu1,PSch}; the typical structure of intersecting families of discrete structures is studied in~\cite{BDDLS}; see also~\cite{abms14} for a description of the typical sum-free set in finite abelian groups.

In contrast to the family of all $n$-vertex triangle-free graphs, which has been well-studied, very little was known about the subfamily consisting of all those that are maximal (under graph inclusion) triangle-free. Note that the size of the family of triangle-free graphs on $[n]$ is at least $2^{n^2/4}$ (all subgraphs of a complete balanced bipartite graph), and at most $2^{n^2/4+o(n^2)}$ by the result of Erd\H{o}s, Kleitman, and Rothschild from 1976. Until recently, it was not even known if the subfamily of maximal triangle-free graphs is significantly smaller. As a first step, Erd\H{o}s suggested the following problem (as stated in~\cite{Simonovits}): determine or estimate the number of maximal triangle-free graphs on $n$ vertices. The following folklore construction shows that there are at least $2^{n^2/8}$  maximal triangle-free graphs on the vertex set $[n]:=\{1,\dots,n\}$. 

\medskip

\noindent\textbf{Lower bound construction.} 
Assume that $n$ is a multiple of $4$. Start with a graph on a vertex set $X\cup Y$ with $|X|=|Y|=n/2$ such that $X$ induces a perfect matching and $Y$ is an independent set (see Figure \ref{fig1a}). For each pair of a matching edge $x_1 x_2$ in $X$ and a vertex $y\in Y$, add exactly one of the edges $x_1 y$ or $x_2 y$. Since there are $n/4$ matching edges in $X$ and $n/2$ vertices in $Y$, we obtain $2^{n^2/8}$ triangle-free graphs. These graphs may not be maximal triangle-free, but since no further edges can be added between $X$ and $Y$, all of these $2^{n^2/8}$ graphs extend to distinct maximal ones.

\begin{figure*}[t!]
    \centering
    \begin{subfigure}[t]{0.4\textwidth}
        \centering
        \includegraphics{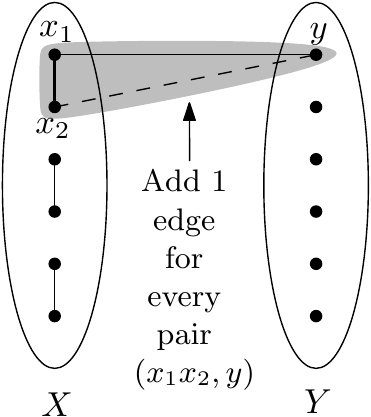}
        \caption{$r=2$}\label{fig1a}
    \end{subfigure}%
    \hspace{1cm}
        \begin{subfigure}[t]{0.4\textwidth}
        \centering
        \includegraphics{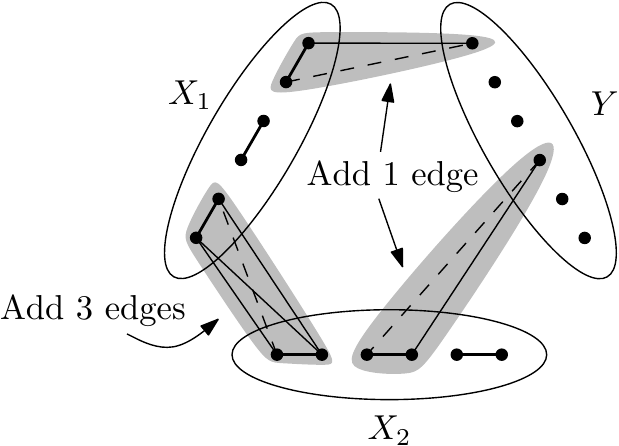}
        \caption{$r=3$}\label{fig1b}
    \end{subfigure}
   \caption{Lower bound contruction for maximal $K_{r+1}$-free graphs.}
\end{figure*}

\medskip

Balogh and Pet\v{r}\'{i}\v{c}kov\'{a}~\cite{BP} recently proved a matching upper bound, that the number of maximal triangle-free graphs on vertex set $[n]$ is at most $2^{n^2/8+o(n^2)}$. Now that the counting problem is resolved, one would naturally ask how do most of the maximal triangle-free graphs look, i.e. what is their typical structure. Our main result provides an answer to this question.  
\begin{theorem}\label{thm_main}
For almost every maximal triangle-free graph $G$ on $[n]$, there is a vertex partition $X\cup Y$ such that $G[X]$ is a perfect matching and $Y$ is an independent set.
\end{theorem}
It is worth mentioning that once a maximal triangle-free graph has the above partition $X\cup Y$, then there has to be exactly one edge between every matching edge of $X$ and every vertex of $Y$. Thus Theorem~\ref{thm_main} implies that almost all maximal triangle-free graphs have the same structure as the graphs in the lower bound construction above. Furthermore, our proof yields that the number of maximal triangle-free graphs without the desired structure is exponentially smaller than the number of maximal triangle-free graphs: Let $\MTF(n)$ denote the set of all maximal triangle-free graphs on $[n]$, and $\cG(n)$ denote the family of graphs from $\MTF(n)$ that admit a vertex partition such that one part induces a perfect matching and the other is an independent set. Then there exists an absolute constant $c>0$ such that for $n$ sufficiently large, $|\MTF(n)-\cG(n)|\leq 2^{-cn}|\MTF(n)|$.

\medskip

It would be interesting to have similar results for $\cM_r(n)$, the number of maximal $K_r$-free graphs on $[n]$. Alon pointed out that if the number of maximal $K_r$-free graphs is $2^{c_r n^2 +o(n^2)}$, then $c_r$ is monotone increasing in $r$, though not necessarily strictly monotone. For the lower bound, a discussion with Alon and \L uczak led to the following construction that gives $2^{(1-1/r+o(1))n^2/4}$ maximal $K_{r+1}$-free graphs: Assume that $n$ is a multiple of $2r$. Partition the vertex set $[n]$ into $r$ equal classes $X_1,\dots, X_{r-1},Y$, and place a perfect matching into each of $X_1,\dots, X_{r-1}$ (see Figure \ref{fig1b}). Between the classes we have the following connection rule:
between the vertices of two matching edges from different classes $X_i$ and $X_j$ place exactly three edges, and between a vertex in $Y$ and a matching edge in $X_i$ put exactly one edge. For the upper bound, by Erd\H{o}s, Frankl and R\"{o}dl~\cite{EFR}, $\cM_{r+1}(n)\leq 2^{(1-1/r+o(1))n^2/2}$. A slightly improved bound is given in \cite{BP}: For every $r$ there is $\ep(r)>0$ such that $|\cM_{r+1}(n)|\leq 2^{(1-1/r-\ep(r))n^2/2}$ for $n$ sufficiently large. We suspect that the lower bound is the ``correct value'', i.e. that $|\cM_{r+1}(n)|= 2^{(1-1/r+o(1))n^2/4}$.

\medskip

\noindent\textbf{Related problem.} 
There is a surprising connection between the family of maximal triangle-free graphs and the family of maximal sum-free sets in $[n]$. More recently, Balogh, Liu, Sharifzadeh and Treglown~\cite{BLST} proved that the number of maximal sum-free sets in $[n]$ is $2^{(1+o(1))n/4}$, settling a conjecture of Cameron and Erd\H{o}s. Although neither of the results imply one another, the methods in both of the papers fall in the same general framework, in which a rough structure of the family is obtained first using appropriate container lemma and removal lemma. These are Theorems~\ref{thm1} and~\ref{thm2} in this paper, and a group removal lemma of Green~\cite{G} and a granular theorem of Green and Ruzsa~\cite{G-R} in the sum-free case. Both problems can then be translated into bounding the number of maximal independent sets in some auxiliary link graphs. In particular, one of the tools here (Lemma~\ref{lem-mis-p3}) is also utilized in~\cite{BLST2} to give an asymptotic of the number of maximal sum-free sets in $[n]$.

\medskip

\noindent\textbf{Organization.}
We first introduce all the tools in Section~\ref{s_Tools}, then we prove Lemma~\ref{lem-asymp}, the asymptotic version of Theorem~\ref{thm_main}, in Section~\ref{s_Asymptotic}. Using this asymptotic result we prove Theorem~\ref{thm_main} in Section~\ref{s_Proof}. 

\medskip

\noindent\textbf{Notation.} 
For a graph $G$, denote by $|G|$ the number of vertices in $G$. An $n$-vertex graph $G$ is \emph{$t$-close to bipartite} if $G$ can be made bipartite by removing at most $t$ edges.  Denote by $P_k$ the path on $k$ vertices. Write $\MIS(G)$ for the number of maximal independent sets in $G$. The \emph{Cartesian product} $G \square H$ of graphs $G$ and $H$ is a graph with vertex set $V(G)\times V(H)$ such that two vertices $(u,u')$ and $(v,v')$ are adjacent if and only if either $u=v$ and $u'v'\in E(H)$, or  $u'=v'$ and $uv\in E(G)$. For a fixed graph $G$, let
$N(v)$ be the set of neighbors of a vertex $v$ in $G$, and let $d(v):=|N_{G}(v)|$ and $\Gamma(v):=N(v)\cup\{v\}$. For $v\in V(G)$ and $X\subseteq V(G)$, denote by $N_X(v)$ the set of all neighbors of $v$ in $X$ (i.e. $N_X(v)=N(v)\cap X$), and let $d_X(v):=|N_X(v)|$. Denote by $\De(X)$ the maximum degree of the induced subgraph $G[X]$. Given a vertex partition $V=X_1\cup X_2$, edges with one endpoint in $X_1$ and the other endpoint in $X_2$ are \emph{$[X_1,X_2]$-edges}. A vertex cut $V=X\cup Y$ is a \emph{max-cut} if the number of $[X,Y]$-edges is not smaller than the size of any other cut. The \emph{inner neighbors} of a vertex $v$ are its neighbors in the same partite set as $v$ (i.e.~$N_{X_i}(v)$ if $v\in X_i$). The \emph{inner degree} of a vertex is the number of its inner neighbors. 
 We say that a family $\cF$ of maximal triangle-free graphs is \emph{negligible} if there exists an absolute constant $C>0$ such that $|\cF|<2^{-Cn}|\cM_3(n)|$.

\section{Tools}\label{s_Tools}
Our first tool is  a corollary of recent powerful counting theorems of Balogh-Morris-Samotij~\cite[Theorem 2.2.]{Balogh}, and Saxton-Thomason~\cite{Saxton}.

\begin{theorem}\label{thm1} 
For all $\delta>0$ there is $c=c(\delta)>0$ such that there is a family $\cF$ of at most $2^{c\cdot \log n\cdot n^{3/2}}$ graphs on $[n]$, each containing at most $\delta n^3$ triangles, such that for every triangle-free graph $G$ on $[n]$ there is an $F\in \cF$ such that $G \subseteq  F$, where $n$ is sufficiently large.
\end{theorem}

\noindent
The graphs in $\cF$ in the above theorem will be referred to as \emph{containers}. A weaker version of Theorem~\ref{thm1}, which can be concluded from the Szemer\'{e}di Regularity Lemma, could be used instead of Theorem~\ref{thm1} here. The only difference is that the upper bound on the size of $\cF$ is $2^{o(n^2)}$.

We need two well-known results. The first is the Ruzsa-Szemer\'{e}di triangle-removal lemma~\cite{Ruzsa} and the second is the Erd\H{o}s-Simonovits stability theorem \cite{ES}:
\begin{theorem}\label{thm2}
For every $\ep>0$ there exists $\delta=\delta(\ep)>0$ and $n_0(\ep)>0$ such that any graph $G$ on $n>n_0(\ep)$ vertices with at most $\delta n^3$ triangles can be made triangle-free by removing at most $\ep n^2$ edges.
\end{theorem}

\begin{theorem}\label{thm3}
For every $\ep >0$ there exists $\delta=\delta(\ep)>0$ and $n_0(\ep)>0$ such that every triangle-free graph $G$ on $n>n_0(\ep)$ vertices with at least $\frac{n^2}{4}-\delta n^2$ edges can be made bipartite by removing at most $\ep n^2$ edges.
\end{theorem}

We also need the following lemma, which is an extension of results of Moon-Moser~\cite{MM} and Hujter-Tuza~\cite{HT}.

\begin{lemma}\label{lem-mis-p3}
Let $G$ be an $n$-vertex triangle-free graph. If $G$ contains at least $k$ vertex-disjoint $P_3$'s, then 
\begin{eqnarray}\label{eq-mis}
\MIS(G)\le 2^{\frac{n}{2}-\frac{k}{25}}.
\end{eqnarray}
\end{lemma}

\begin{proof}
The proof is by induction on $n$. The base case of the induction is $n=1$ with $k=0$, for which $\MIS(G)=1\leq 2^{\frac{1}{2}-\frac{0}{25}}$.

For the inductive step, let $G$ be a triangle-free graph on $n\geq 2$ vertices with $k$ vertex-disjoint $P_3$'s, and let $v$ be any vertex in $G$. Observe that $\MIS(G-\Gamma(v))$ is the number of maximal independent sets containing $v$, and that $\MIS(G-\{v\})$ bounds from above the number of maximal independent sets not containing $v$. Therefore, 
$$\MIS(G)\le \MIS(G-\{v\})+\MIS(G-\Gamma(v)).$$
If $G$ has $k$ vertex-disjoint $P_3$'s, then $G-\Gamma(v)$ has at least $k-d(v)$ vertex-disjoint $P_3$'s, and so, by the induction hypothesis,

$$\MIS(G) \leq 2^{\frac{n-1}{2}-\frac{k-1}{25}}+ 2^{\frac{n-(d(v)+1)}{2}-\frac{k-d(v)}{25}}\le 2^{\frac{n}{2}-\frac{k}{25}}\left(2^{-\frac{1}{2}+\frac{1}{25}}+2^{-\frac{d(v)+1}{2}+\frac{d(v)}{25}}\right).$$
The function $f(x)=2^{-\frac{1}{2}+\frac{1}{25}}+2^{-\frac{x+1}{2}+\frac{x}{25}}$ is a decreasing function with $f(3)\approx 0.9987<1$. So, if there exists a vertex of degree at least $3$ in $G$, then we have $\MIS(G)\leq 2^{\frac{n}{2}-\frac{k}{25}}$ as desired.

It remains to verify \eqref{eq-mis} for graphs with $\Delta(G)\leq 2$. Observe that we can assume that $G$ is connected. Indeed, if  $G_1,\dots, G_l$ are maximal components of $G$, and each of $G_i$ has $n_i$ vertices and $k_i$ vertex-disjoint $P_3$'s, then $$\MIS(G)=\prod_{i} \MIS(G_i) \le \prod_i 2^{\frac{n_i}{2}-\frac{k_i}{25}}=2^{\sum_i \frac{n_i}{2}-\sum_i\frac{k_i}{25}}=2^{\frac{n}{2}-\frac{k}{25}}.$$ 

Every connected graph with $\Delta(G)\leq 2$ and $n\geq 2$ vertices is either a path or a cycle. 
 Suppose first that $G$ is a path $P_n$. We have $\MIS(P_2)=2\leq 2^{\frac{2}{2}-\frac{0}{25}}$, $\MIS(P_3)=2\leq 2^{\frac{3}{2}-\frac{1}{25}}$. By F\"uredi~\cite[Example~1.1]{Furedi}, 
 $\MIS(P_n)= \MIS(P_{n-2})+\MIS(P_{n-3})$ for all $n\geq 4$. By the induction hypothesis thus
\begin{eqnarray*}
\MIS(P_n)\leq  2^{\frac{n-2}{2}-\frac{k-1}{25}}+2^{\frac{n-3}{2}-\frac{k-1}{25}}\le 2^{\frac{n}{2}-\frac{k}{25}}\left(2^{-1+\frac{1}{25}}+2^{-\frac{3}{2}+\frac{1}{25}}\right)\leq 2^{\frac{n}{2}-\frac{k}{25}}.
\end{eqnarray*}
Let now $G$ be a cycle $C_n$. We have $\MIS(C_4)=2\leq 2^{4/2-1/25}$ and $\MIS(C_5)=5\leq 2^{5/2-1/25}$. By F\"uredi~\cite[Example~1.2]{Furedi},  $\MIS(C_n)=\MIS(C_{n-2})+\MIS(C_{n-3})$ for all $n\ge 6$. Therefore, by the induction hypothesis,
$$\MIS(C_n)\leq 2^{\frac{n-2}{2}-\frac{k-1}{25}}+2^{\frac{n-3}{2}-\frac{k-1}{25}}\le 2^{\frac{n}{2}-\frac{k}{25}}.$$
\end{proof}
\begin{rmk}
A disjoint union of $C_5$'s and a matching shows that the constant $c$ for which $\MIS(G)\le 2^{\frac{n}{2}-\frac{k}{c}}$ in Lemma~\ref{lem-mis-p3} cannot be smaller than $5.6$.
\end{rmk}

\section{Asymptotic result}\label{s_Asymptotic}
In this section we prove an asymptotic version of Theorem~
\ref{thm_main}:
\begin{lemma}\label{lem-asymp} Fix any $\ga>0$. Almost every maximal triangle-free graph $G$ on the vertex set $[n]$ satisfies the following: for any max-cut $V(G)=X\cup Y$, there exist $X'\subseteq X$ and $Y'\subseteq Y$ such that 

(i) $|X'|\leq \ga n$ and $G[X\setm X']$ is an induced perfect matching, and

(ii) $|Y'|\leq \ga n$ and $Y\setm Y'$ is an independent set.
\end{lemma}

The outline of the proof is as follows. We observe that every maximal triangle-free graph $G$ on $[n]$ can be built in the following three steps. 
\begin{enumerate}
\setlength{\itemsep}{1pt}
  \setlength{\parskip}{0pt}
  \setlength{\parsep}{0pt}
\item[(S1)] Choose a max-cut $X\cup Y$ for $G$.
\item[(S2)] Choose triangle-free graphs $S$ and $T$ on the vertex sets $X$ and $Y$, respectively. 
\item[(S3)] Extend $S\cup T$ to a maximal triangle-free graph by adding edges between $X$ and $Y$. 
\end{enumerate}
We give an upper bound on the number of choices for each step. First, there are at most $2^n$ ways to fix a max-cut $X\cup Y$ in (S1). For (S2), we show (Lemma~\ref{lem-large}) that almost all maximal triangle-free graphs on $[n]$ are $o(n^2)$-close to bipartite, which implies that the number of choices for most of these graphs in (S2) is at most $2^{o(n^2)}$. For fixed $X,Y,S,T$, we bound, using Claim~\ref{cl_MIS}, the number of choices in (S3) by the number of maximal independent sets in some auxiliary link graph $L$. This enables us to use Lemma~\ref{lem-mis-p3} to force the desired structure on $S$ and $T$.

\begin{defn} [\textbf{Link graph}]
Given edge-disjoint graphs $A$ and $S$ on $[n]$, define 
the \emph{link graph} $L:=L_{S}[A]$ of $S$ on $A$ as follows:
$$
\begin{array}{lllll}
V(L):=E(A) \quad \mbox{ and }\quad
E(L):=\{a_1 a_2: \exists s\in E(S)\mbox{ such that }  \{a_1,a_2,s\} \mbox{ forms a triangle}  \}.
\end{array}
$$
\end{defn}

\begin{claim}\label{cl_triangle_free}
If $A$ and $S$ are triangle-free, then $L_S[A]$ is triangle-free.
\end{claim}
\begin{proof}
Indeed, otherwise there exist $a_1,a_2,a_3\in E(A)$ and $s_1,s_2,s_3\in E(S)$ such that the $3$-sets $\{a_1,a_2,s_1\}$, $\{a_2,a_3,s_2\}$, and $\{a_1,a_3,s_3\}$ span triangles. Since $A$ is triangle-free, the edges $a_1,a_2,a_3$ share a common endpoint, and $\{s_1,s_2,s_3\}$ spans a triangle. This is a contradiction since $S$ is triangle-free. 
\end{proof}

\begin{claim}\label{cl_MIS}
Let $S$ and $A$ be two edge-disjoint triangle-free graphs on $[n]$ such that there is no triangle $\{a,s_1,s_2\}$ in $S\cup A$ with $a\in E(A)$ and $s_1, s_2\in E(S)$. Then the number of maximal triangle-free subgraphs of $S\cup A$ containing $S$ is at most $\MIS(L_S[A])$.
\end{claim}

\begin{proof}
Let $G$ be a maximal triangle-free subgraph of $S\cup A$ that contains $S$. We show that $E(G)\cap E(A)$ spans a maximal independent set in $L:=L_S[A]$. Clearly, $E(G)\cap E(A)$ spans an independent set in $L$ because otherwise there would be a triangle in $G$. Suppose that $E(G)\cap E(A)$ is not a maximal independent set in $L$. Then there is $a_1\in E(A)- E(G)$ such that, for any two edges $a_2\in E(A)\cap E(G)$ and $s\in E(S)$, $\{a_1,a_2,s\}$ does not form a triangle. By our assumption, there is no triangle $\{a_1,a_2,a_3\}$ with $a_2,a_3\in E(A)$ and no triangle $\{a_1,s_1,s_2\}$ with $s_1,s_2\in E(S)$. Therefore, $G\cup \{a_1\}$ is triangle-free, contradicting the maximality of $G$.
\end{proof}

We fix the following parameters that will be used throughout the rest of the paper. Let $\ga,\beta,\ep,\ep'>0$ be sufficiently small constants satisfying the following hierachy:
\begin{eqnarray}\label{eq-para}
\ep'\ll\de_{2.3}(\ep)\ll\ep\ll\beta\ll\de_{2.3}(\ga^3)\ll\ga\ll 1,
\end{eqnarray}
where $\de_{2.3}(x)>0$ is the constant returned from Theorem~\ref{thm3} with input $x$. The notation $x\ll y$ above means that $x$ is a sufficiently small function of $y$ to satisfy some inequalities in the proof. In the following proof, $\de_{2.2}(x)$ is  the constant returned from Theorem~\ref{thm2} with input $x$, and in the rest of the paper, we shall always assume that $n$ is sufficiently large, even when this is not explicitly stated.

\begin{lemma}\label{lem-large}
Almost all maximal triangle-free graphs on $[n]$ are $2\ep n^2$-close to bipartite.
\end{lemma}
\begin{proof}
Let $\cF$ be the family of graphs obtained from Theorem~\ref{thm1} using $\de_{2.2}(\ep')$. Then every triangle-free graph on $[n]$ is a subgraph of some container $F\in\cF$.

We first show that the family of maximal triangle-free graphs in small containers is negligible. Consider a container $F\in \cF$ with $e(F)\leq n^2/4-6\ep' n^2$. Since $F$ contains at most $\de_{2.2}(\ep') n^3$ triangles, by Theorem~\ref{thm2}, we can find $A$ and $B$, subgraphs of $F$, such that $F=A\cup B$,  where $A$ is triangle-free, and $e(B)\le \ep' n^2$. For each $F\in\cF$, fix such a pair $(A,B)$. Then every maximal triangle-free graph in $F$ can be built in two steps:
\begin{enumerate}
\setlength{\itemsep}{1pt}
  \setlength{\parskip}{0pt}
  \setlength{\parsep}{0pt}
  \item[(i)] Choose a triangle-free $S\subseteq B$;
\item[(ii)] Extend $S$ in $A$ to a maximal triangle-free graph.
\end{enumerate}
The number of choices in (i) is at most $2^{e(B)}\le 2^{\ep' n^2}$. Let $L:=L_S[A]$ be the link graph of $S$ on $A$. By Claim~\ref{cl_triangle_free}, $L$ is triangle-free. Claim~\ref{cl_MIS} implies that the number of maximal triangle-free graphs in $S\cup A$ containing $S$ (i.e.~the number of extensions in (ii)) is at most $\MIS(L)$. Thus, by Lemma~\ref{lem-mis-p3},
$$ \MIS(L) \le 2^{|A|/2}\le 2^{n^2/8-3\ep' n^2}. $$
Therefore, the number of maximal triangle-free graphs in small containers is at most
$$|\cF|\cdot 2^{\ep' n^2}\cdot 2^{n^2/8-3\ep' n^2}\le 2^{n^2/8-\ep' n^2}.$$

From now on, we may consider only maximal triangle-free graphs contained in containers of size at least $n^2/4-6\ep' n^2$. Let $F$ be any large container. Recall that by Theorem~\ref{thm2}, $F=A\cup B$, where $A$ is triangle-free with $e(A)\ge n^2/4-7\ep' n^2$ and $e(B)\le \ep' n^2$. Since $\ep'\ll \de_{2.3}(\ep)$, by Theorem~\ref{thm3}, $A$ can be made bipartite by removing at most $\ep n^2$ edges. Since $\ep'\ll \ep$, $F$ can be made bipartite by removing at most $(\ep'+\ep)n^2\le 2\ep n^2$ edges. Therefore, every maximal triangle-free graphs contained in $F$ is $2\ep n^2$-close to bipartite.
\end{proof}

Fix $X,Y,S,T$ as in steps (S1) and (S2). Let $A$ be the complete bipartite graph with parts $X$ and $Y$. By Claim~\ref{cl_MIS}, the number of ways to extend $S\cup T$ in (S3) is at most $\MIS(L_{S\cup T}[A])$. The number of ways to fix $X$ and $Y$ is at most $2^n$, and by Lemma~\ref{lem-large}, the number of ways to fix $S$ and $T$ is at most $\binom{n^2}{2\ep n^2}$. It follows that if $\MIS(L_{S\cup T}[A])$ is smaller than $2^{n^2/8-cn^2}$ for some $c\gg\ep$, then the family of maximal triangle-free graphs with such $(X,Y,S,T)$ is negligible.

\begin{claim}\label{claim_L}
$L_{S\cup T}[A]=S\square T$.
\end{claim}

\begin{proof}
Note that $V(L_{S\cup T}[A])=E(A)=\{(x,y):x\in X, y\in Y\}=V(S\square T)$.
 Using the definition of the 
Cartesian product, $(x,y)$ and $(x',y')$ are adjacent in $S\square T$ if and only if $x=x'$ and $\{y,y'\}\in E(T)$, or $y=y'$ and $\{x,x'\}\in E(S)$, i.e. if and only if $\{x=x',y,y'\}$ or $\{x,x',y=y'\}$ form a triangle in $S\cup A$. But by the definition of $L_{S\cup T}[A]$, this is exactly when $(x,y)$ and $(x',y')$ are adjacent in $L_{S\cup T}[A]$. 
\end{proof}

Claim~\ref{claim_L} allows us to rule out certain structures of $S$ and $T$ since, by Lemma~\ref{lem-mis-p3}, if $S\square T$ has many vertex disjoint $P_3$'s then the number of maximal-triangle free graphs with $S=G[X]$ and $T=G[Y]$ is much smaller than $2^{n^2/8}$.

\begin{claim} \label{cl-XY}
For almost all maximal triangle-free $n$-vertex graphs $G$ with a max-cut $X\cup Y$, 

(i) $|X|, |Y|\geq n/2-\beta n$, and 

(ii) $\De(X), \De(Y)\le \beta n$. 
\end{claim}

\begin{proof}

Let $G$ be a maximal triangle-free graph with a max-cut $X\cup Y$. By Lemma~\ref{lem-large}, almost all maximal triangle-free graphs are $2\ep n^2$-close to bipartite, which implies that the number of choices for $G[X]$ and $G[Y]$ is at most ${n^2\choose 2\ep n^2}$. Denote by $A$ the complete bipartite graph with partite sets $X$ and $Y$.

For (i), suppose that $|X|\leq n/2-\beta n$. Then $|X||Y|\leq n^2/4-\beta^2 n^2$, and for any fixed $S$ on $X$ and $T$ on $Y$, Lemma~\ref{lem-mis-p3} implies $\MIS(L_{S\cup T}[A])\leq 2^{n^2/8-\beta^2 n^2/2}$. Since $\beta\gg \ep$, it follows from the discussion before Claim~\ref{claim_L} that the family of maximal triangle-free graphs with such max-cut $X\cup Y$ is negligible.

For (ii), suppose that $G$ has a vertex $x\in X$ of inner degree at least $\beta n$. Since $X\cup Y$ is a max-cut, $|N_Y(x)|\ge|N_X(x)|\ge\beta n$. Since $G$ is triangle-free, there is no edge in between $N_X(x)$ and $N_Y(x)$. Let $A'\subseteq A$ be a graph formed by deleting all edges between $N_X(x)$ and $N_Y(y)$ from $A$. Define a link graph $L':=L_{S\cup T}[A']$ of $S\cup T$ on $A'$. In this case, the number of choices for (S3) is at most $\MIS(L')$. Since  $L'$ is triangle-free (Claim~\ref{cl_triangle_free}) and $|L'|=e(A')\leq |X| |Y|-|N_X(x)||N_Y(x)|\leq \frac{n^2}{4}-\beta^2 n^2$, it follows from Lemma~\ref{lem-mis-p3} that $$\MIS(L')\le 2^{|L'|/2}\le 2^{n^2/8-\beta^2n^2/2}.$$
\end{proof}

\begin{proof}[Proof of Lemma~\ref{lem-asymp}]
First, we show that for almost every maximal triangle-free graph $G$ on $[n]$ with max-cut $X\cup Y$ and with $G[X]=S$ and $G[Y]=T$, there are very few vertex-disjoint $P_3$'s in $S\cup T$. Suppose that there exist $\beta n$ vertex-disjoint $P_3$'s in $S$ or in $T$, say in $S$. Since $L_{S\cup T}[A]=S\square T$ by Claim~\ref{claim_L}, and for each of the $\beta n$ vertex-disjoint $P_3$'s in $S$ we obtain $|T|$ vertex-disjoint $P_3$'s in $S\square T$, the number of vertex-disjoint $P_3$'s in $L_{S\cup T}[A]$ is at least $\beta n|T|=\beta n|Y|$. By Claim~\ref{cl-XY}(i), $\beta n |Y|\geq \beta n (n/2-\beta n)\ge \beta n^2/3$. Then by Lemma~\ref{lem-mis-p3},
$$\MIS(L_{S\cup T}[A])\leq 2^{|S\square T|/2-\beta n^2/75}\leq2^{n^2/8-\beta n^2/75}.$$
Since $\beta\gg \ep$, the family of maximal triangle-free graphs with such $(X,Y,S,T)$ is negligible. Hence, for almost every maximal triangle-free graph $G$ with some $(X,Y,S,T)$, we can find some induced subgraphs $S'\subseteq S$ and $T'\subseteq T$ with $|S'|\le 3\beta n$ and $|T'|\le 3\beta n$ such that both $S-S'$ and $T-T'$ are $P_3$-free. This implies that each of $S-S'$ and $T-T'$ is a union of a matching and an independent set. 

\begin{figure*}[t!]
    \centering
    \begin{subfigure}[t]{0.4\textwidth}
        \centering
        \includegraphics{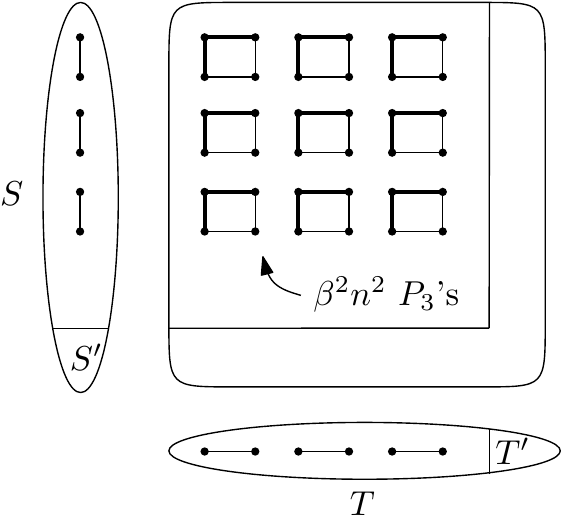}
        \caption{}\label{fig2a}
    \end{subfigure}%
    \hspace{2cm}
        \begin{subfigure}[t]{0.4\textwidth}
        \centering
        \includegraphics{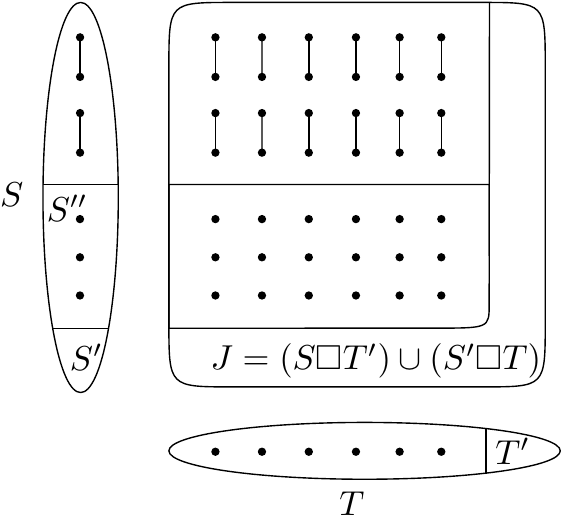}
        \caption{}\label{fig2b}
    \end{subfigure}
   \caption{Forbidden structures in $S$ and $T$.}
\end{figure*}

Next, we show that at most one of the graphs $S$ and $T$ can have a large matching. Suppose both $S$ and $T$ have a matching of size at least $\beta n$, then there are at least $\beta^2 n^2$ vertex-disjoint $C_4$'s in $S\square T$, each of which contains a copy of $P_3$ (see Figure~\ref{fig2a}). It follows that the family of such graphs is negligible since $\MIS(L_{S\cup T}[A])\le 2^{n^2/8-\beta^2n^2/25}$ and $\beta\gg \ep$. Hence, we can assume that all but $2 \beta n$ vertices in $T$ form an independent set. Redefine $T'$ so that $|T'|\le 2\beta n$ and $V(T-T')$ is an independent set.
%

Lastly, we show that there are very few isolated vertices in the graph $S-S'$. Suppose that there are $\ga n/2$ isolated vertices in $S-S'$, spanning a subgraph $S''$ of $S$. We count $\MIS(S\square T)$ as follows.
Let $J:=(S\square T') \cup (S'\square T)$ and $L':=S\square T- J$. Every maximal independent set in $S\square T$ can be built by 
\begin{enumerate}
\setlength{\itemsep}{1pt}
  \setlength{\parskip}{0pt}
  \setlength{\parsep}{0pt}
\item[(i)] choosing an independent set in $J$, and 
\item[(ii)] extending it to a maximal independent set in $L'$. 
\end{enumerate}
Since $|J|\leq |S'||T|+|T'||S|\leq 3\beta n \cdot n+ 2\beta n \cdot n=5\beta n^2$, there are at most $2^{|J|}=2^{5\beta n^2}$ choices for (i). Note that $L'$ consists of isolated vertices from $S''\square(T- T')$ and an induced matching from $(S-S'-S'')\square (T-T')$ (see Figure~\ref{fig2b}). Thus the number of extensions in (ii) is at most $\MIS((S-S'-S'')\square (T-T'))$. The graph $(S-S'-S'')\square (T-T')$ is a perfect matching with 
$$\frac{1}{2}|S-S'-S''||T-T'|\leq \frac{1}{2}|S-S''||T|\leq\frac{1}{2}\left(|S|-\frac{\ga n}{2}\right)(n-|S|) \leq \frac{1}{2}\left(\frac{n}{2}-\frac{\ga n}{4}\right)^2\le \frac{n^2}{8}-\frac{\ga n^2}{16}$$ 
edges, and so choosing one vertex for each matching edge gives at most $2^{n^2/8-\ga n^2/16}$ maximal independent sets. Since $\beta\ll\ga$, it follows that $\MIS(S\square T)\leq 2^{5\beta n^2}\cdot 2^{n^2/8-\ga n^2/16}
\le 2^{n^2/8-\ga n^2/17}$. Thus, such family of maximal triangle-free graphs is negligible, and we may assume that $|S''|\le\ga n/2$.

 The statement of Lemma~\ref{lem-asymp} follows by setting $X':=V(S'\cup S'')$ and $Y':=V(T')$. Indeed, $|X'|\leq 3\beta n+ \gamma n/2\le \gamma n$, $|Y'|\leq 2\beta n\le \gamma n$, $G[X-X']=S-S'-S''$ is a perfect matching, and $Y-Y'=V(T)-V(T')$ is an independent set.

\end{proof}

\section{Proof of Theorem~\ref{thm_main}}\label{s_Proof}
For the proof of Theorem~\ref{thm_main}, we need to introduce several classes of graphs on the vertex set $V=[n]$. Recall the hierarchy of parameters fixed in Section~\ref{s_Asymptotic}:
\begin{eqnarray}
\ep'\ll\de_{2.3}(\ep)\ll\ep\ll\beta\ll\de_{2.3}(\ga^3)\ll\ga\ll 1,
\end{eqnarray}

\begin{defn}
Fix a vertex partition $V=X\cup Y$, a perfect matching $M$ on the vertex set $X$ (in case $|X|$ is odd, $M$ is an almost perfect matching covering all but one vertex of $X$), and non-negative integers $r$, $s$ and $t$. 

\medskip

\noindent 1. Denote by $\cB(X,Y,M,s,t)$ the class of maximal triangle-free graphs $G$ with max-cut $X\cup Y$ satisfying the following three conditions:

(i) The subgraph $G[X]$ has a maximum matching $M'\subseteq M$ covering all but at most $\ga n$ vertices in $X$; 

(ii) The size of a largest family of vertex-disjoint $P_3$'s in $S:=G[X]$ is $s$;

(iii) The size of a maximum matching in $T:=G[Y]$ is $t$.

\medskip

\noindent 2. Denote by $\cB(X,Y,M, r)\subseteq \cB(X,Y,M,0,0)$ the subclass consisting of all graphs in $\cB(X,Y,M,0,0)$ with exactly $r$ isolated vertices in $G[X]$.

\medskip

\noindent 3. When $|X|$ is even, denote by $\cG(X,Y,M)$ the class of all maximal triangle-free graphs $G$ with max-cut $X\cup Y$,  $G[X]=M$, and $Y$ an independent set. 

\medskip

\noindent 4. When $|X|$ is even, denote by $\cH(X,Y,M)$ the class of maximal triangle-free graphs $G$ that are constructed as follows: 

(P1) Add $M$ to $X$; 

(P2) For every edge $x_1x_2\in M$ and every vertex $y\in Y$, add either the edge $x_1y$ or $x_2y$;

(P3) Extend each of the $2^{|X||Y|/2}$ resulting graphs to a maximal triangle-free graph by adding edges in $X$ and/or $Y$.
\end{defn}

By Lemmas~\ref{lem-asymp},~\ref{lem-large} and Claim~\ref{cl-XY}, throughout the rest of the proof, we may only consider maximal triangle-free graphs in $\bigcup_{X,Y,M,s,t}\cB(X,Y,M,s,t)$ that are $\beta n^2$-close to bipartite, $|X|,|Y|\ge n/2-\beta n$ and $\De(X),\De(Y)\le \beta n$. We may further assume from the proof of Lemma~\ref{lem-asymp} that $s,t\le \beta n$.

Notice that graphs from $\cG(X,Y,M)=\cB(X,Y,M,0)$ are precisely those with the desired structure. We will show that the number of graphs without the desired structure is exponentially smaller. The set of ``bad'' graphs consists of the following two types:

(i) when $|X|$ is even, $\bigcup_{s,t}\cB(X,Y,M,s,t)\setm \cB(X,Y,M,0)$;

(ii) when $|X|$ is odd, $\bigcup_{s,t}\cB(X,Y,M,s,t)$.

Fix an arbitrary choice of $(X,Y,M)$. For simplicity, let $\cB(s,t):=\cB(X,Y,M,s,t)$ and $\cB(r):=\cB(X,Y,M,r)$. Let $A$ be the complete bipartite graph with parts $X$ and $Y$. 

\begin{figure*}[t!]
    \centering
    \begin{subfigure}[t]{0.43\textwidth}
        \centering
        \includegraphics{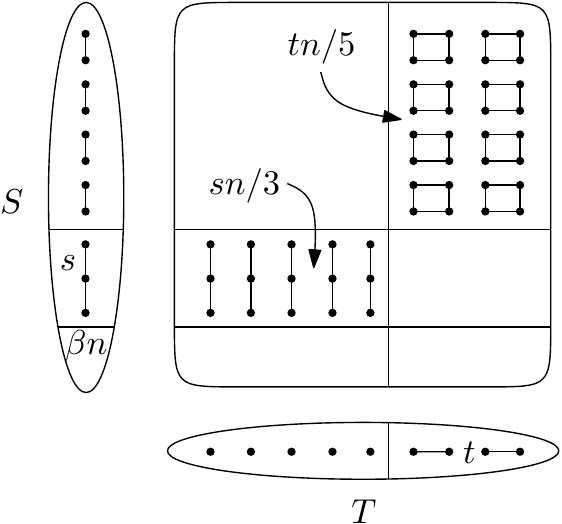}
        \caption{The number of vertex-disjoint $P_3$'s in $S\square T$ is at least $sn/3+tn/5$ (Lemma~\ref{claim-st}).}\label{fig3a}
    \end{subfigure}
    \hspace{0.1\textwidth}
        \begin{subfigure}[t]{0.43\textwidth}
        \centering
        \includegraphics{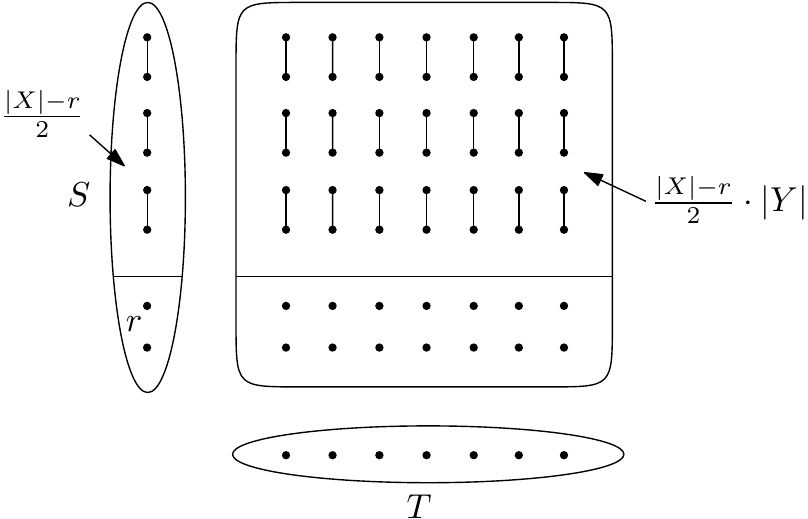}
    \caption{$\MIS(S\square T)\le 2^{(|X|-r)|Y|/2}$ if $s=t=0$ and $X$ has $r$ isolated vertices  (Lemma~\ref{claim-r}).} \label{fig3b}
    \end{subfigure}
   \caption{}
\end{figure*}

\begin{lemma}\label{claim-st}
If $s+t\geq 1$, then $|\cB(s,t)|\le 2^{|X||Y|/2-n/200}.$
\end{lemma}

\begin{proof}
Let $s$ and $t$ be two non-negative integers, at least one of which is nonzero. We first bound the number of ways to choose $S$ and $T$, i.e.~the number of ways to add inner edges. The number of ways to choose the vertex set of the $s$ vertex-disjoint $P_3$'s in $S$ and the $t$ matching edges in $T$ is at most $\binom{n}{3s}\binom{n}{2t}$. Since $\De(X),\De(Y)\le\beta n$, each of the $3s+2t$ chosen vertices has inner degree at most $\beta  n$. Therefore, the number of ways to choose their inner neighbors is at most
$$\binom{n}{\beta  n}^{3s+2t}\leq \left( \left(\frac{e n}{\beta  n}\right)^{\beta  n}\right)^{3s+2t}\leq 2^{\beta \log(e/\beta )\cdot (3s+2t)n}.$$
 
 The number of ways to add the $[X,Y]$-edges is $\MIS(L_{S\cup T}(A))$. We claim that the link graph $L:=L_{S\cup T}(A)=S\square T$ has at least $(s+t)n/5$ vertex-disjoint $P_3$'s. Indeed, recall that $|S|=|T|\ge n/2-\beta n$ and $s,t\le\beta n$, thus in $S\square T$ (see Figure~\ref{fig3a}), we have at least $s(|T|-2t)\geq sn/3$ vertex-disjoint $P_3$'s coming from $s$ vertex-disjoint $P_3$'s in $S$ and at least $\frac{1}{2}(|S|-\beta n-3s)\cdot t \geq tn/5$  vertex-disjoint $P_3$'s coming from the Cartesian product of a matching in $S$ and a matching in $T$. So by Lemma~\ref{lem-mis-p3}, $$\MIS(L)\leq 2^{|X||Y|/2-(s+t)n/125}.$$
Since $s+t\ge 1$ and $\beta $ is sufficiently small,
$$|\cB(s,t)|\leq \binom{n}{3s}\binom{n}{2t}\cdot 2^{\beta \log(e/\beta )\cdot (3s+2t)n}\cdot  2^{|X||Y|/2-(s+t)n/125}\leq 2^{|X||Y|/2-n/200}.$$
\end{proof}

\begin{lemma}\label{claim-r}
If $s=t=0$ and $r\in\mathbb{Z}^+$, then $|\cB(r)|\le 2^{|X||Y|/2-n/6}.$
\end{lemma}
\begin{proof}
By the definition of $\cB(r)$, $X$ consists of $r$ isolated vertices and a matching of size $(|X|-r)/2$, and $Y$ is an independent set. Hence the graph $L_{S\cup T}(A)=S\square T$ consists of a matching of size $(|X|-r)|Y|/2$ and isolated vertices (see Figure~\ref{fig3b}). 
There are at most ${n\choose r}$ ways to pick the isolated vertices in $X$ and at most $\MIS(L_{S\cup T}(A))$ ways to choose the $[X,Y]$-edges. Recall that $|Y|\ge n/2-\beta n$. Thus we have
$$|\cB(r)|\le {n\choose r}\cdot 2^{(|X|-r)|Y|/2}\le 2^{|X||Y|/2+r\log n-rn/5}\le 2^{|X||Y|/2-rn/6} \le 2^{|X||Y|/2-n/6}.$$
\end{proof}

\medskip

\noindent\textbf{Case 1:} $|X|$ is even. For simplicity, denote $\cG:=\cG(X,Y,M)$ and $\cH:=\cH(X,Y,M)$.

\medskip

\begin{figure*}[t!]
\centering
  \includegraphics{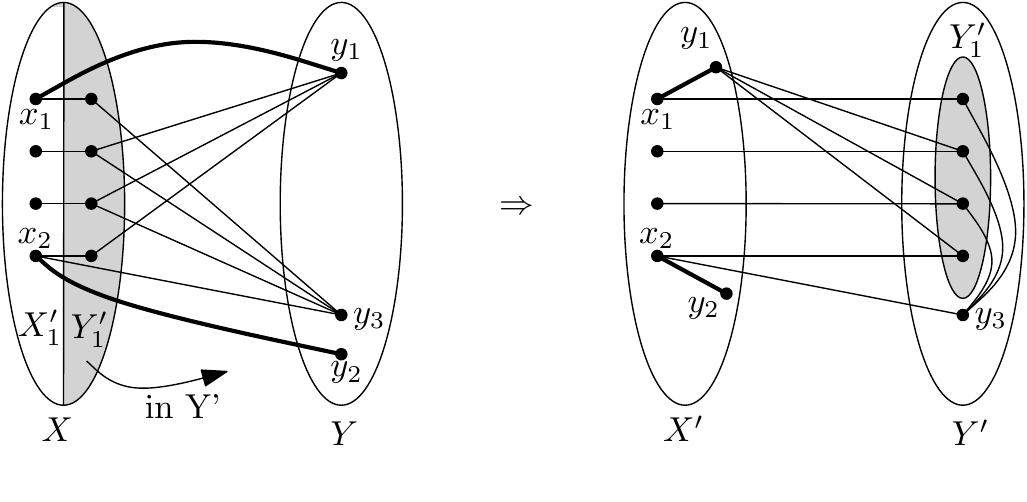}
    \caption{$(X',Y',M')$ is uniquely determined after choosing $x_1y_1\in M'$ (Lemma~\ref{lem-overcount}).}\label{fig4}
\end{figure*}

\begin{lemma}\label{lem-overcount}
An $n$-vertex graph $G$ is in at most $n^2$ different classes $\cG(X,Y,M)$.
\end{lemma}
\begin{proof}
Let $G\in \cG(X,Y,M)$. Recall that $G[X]=M$ and $Y$ is an independent set. Thus $G$ can be in a different class $\cG(X',Y',M')$ if and only if $X'\neq X$, $Y'\neq Y$ and $M'\neq M$. Since $M'\neq M$ and $Y$ is an independent set, there exists an edge $x_1y_1$ in $M'$ with $x_1\in X$ and $y_1\in Y$. There are at most $n^2$ ways to choose such an edge. Since $G$ is a maximal triangle-free graph, every vertex in $Y$ is adjacent to exactly one vertex from each edge in $M$. Let $x_1'$ be the neighbor of $x_1$ in $X$, and set $Y_1':=N_X(y_1)\cup\{x_1'\}-\{x_1\}$ and $X_1'=X\setm Y_1'$. Note that $x_1y_1\in M'$ and $G[X']=M'$ imply $Y'_1\subseteq Y'$. Since $Y'$ is an independent set, it follows that $X'_1\subseteq X'$. 

We claim that for any vertex $x_2\in X'_1$, there is at most one vertex in $Y$ that can serve as its neighbor in $M'$ (see Figure~\ref{fig4}). Suppose to the contrary that there are two such vertices $y_2$ and $y_3$ in $Y$. Then neither of $y_2$ and $y_3$ has neighbors in $X_1'-\{x_2\}$, and so both $y_2$ and $y_3$ are adjacent to all but one (the neighbor of $x_2$) vertex of $Y_1'\subseteq Y'$. If now $x_2y_2\in M'$, then $y_3\in Y'$. But $y_3$ is adjacent to some vertices of $Y'$, which contradicts the independence of $Y'$. In conclusion, after we pick one of the edges of $M'$ with exactly one end in $X$ and one end in $Y$, since the graph $G$ is labeled, the rest of $X'$, $Y'$ and $M'$ is uniquely determined. 
\end{proof}

By Lemma~\ref{lem-overcount}, it is sufficient to show that for any choice of $(X,Y,M)$ with $|X|$ even, 
\begin{eqnarray}\label{eq-ratio}
\frac{|\bigcup_{s,t}\cB(X,Y,M,s,t)\setm \cB(X,Y,M,0)|}{|\cG(X,Y,M)|}\le 2^{-n/300}.
\end{eqnarray}

\begin{lemma}\label{claim-g}
We have $|\cG|\ge (1+o(1))2^{|X||Y|/2}.$
\end{lemma}
\begin{proof}
Recall that $|X|,|Y|\ge n/2-\beta n$, and therefore $|\cH|= 2^{|X||Y|/2}\gg 2^{n^2/8-\beta n^2}$. Running the same proof as Lemma~\ref{lem-large} (start the proof by invoking Theorem~\ref{thm1} with $\de_{2.2}(\beta)$, replace $\ep'$ by $\beta$ and $\ep$ by $\ga^3$) implies that almost all graphs in $\cH$ are $2\ga^3 n^2$-close to bipartite. Let $\cH'\subseteq \cH$ be the subfamily consisting of all those that are $2\ga^3 n^2$-close to bipartite. Then it is sufficient to show $|\cH'\setm \cG|=o(2^{|X||Y|/2})$. There are two types of graphs in $\cH'\setm\cG$: 

(i) $\cH_1$: those that $X\cup Y$ is not one of its max-cut;

(ii) $\cH_2$: those with $X\cup Y$ as a max-cut, but not maximal after (P2), i.e.~there are inner edges added in $X$ and/or $Y$ in (P3).

We first bound the number of graphs in $\cH_1$. Let $G\in\cH_1$ with a max-cut $X'\cup Y'$ minimizing $|X\triangle X'|$. We may assume that $|X'|,|Y'|\ge n/2-\ga n$ and $\De(X'),\De(Y')\le \ga n$. Indeed, since graphs in $\cH_1$ are $2\ga^3n^2$-close to bipartite, if $|X'|\le n/2-\ga n$ or $\De(X')\ge \ga n$, then the same proof as the proof of Claim~\ref{cl-XY} yields that the number of such graphs in $\cH_1$ is at most 
$${n^2\choose 2\ga^3n^2}\cdot2^{n^2/8-\ga^2n^2/2}\le 2^{2\ga^3n^2\log(e/2\ga^3)}2^{n^2/8-\ga^2n^2/2}\ll 2^{n^2/8-\ga n^2},$$ 
which is exponentially smaller than $|\cH'|=(1+o(1))2^{|X||Y|/2}\gg 2^{n^2/8-\beta n^2}$.

Let $X_1:=X\cap X'$, $X_2:=X\setm X_1$, $Y_1:=Y\cap X'$, and $Y_2:=Y\setm Y_1$ (see Figure~\ref{fig5a}). Since $X\cup Y$ is not a max-cut of $G$, the set $X\triangle X'=Y\triangle Y'=X_2\cup Y_1$ is non-empty. By symmetry, we can assume that $Y_1\ne\emptyset$. Recall that from (P2), for every $y\in Y_1\subseteq X'$, we have $d_X(y)=|X|/2\ge n/4-\beta n/2$. It follows that $|X_2|\ge n/4-2\ga n$, since otherwise $d_{X'}(y)\ge d_{X_1}(y)=d_{X}(y)-|X_2|\ge 3\ga  n/2$, contradicting $\De(X')\le \ga n$. Similarly, we have $|X_1|\ge n/4-2\ga n$. Recall also that $|X|=n-|Y|\le n/2+\beta n$. Thus for $i=1,2$, 
$$|X_i|=|X|-|X_{3-i}|\le \frac{n}{2}+\beta n-\left(\frac{n}{4}-2\ga n\right)\le \frac{n}{4}+3\ga n.$$
Therefore, for $i=1,2$, every vertex $y\in Y_i$ is adjacent to at most $\ga n$ vertices in $X_i$ and all but at most 
$$|X_{3-i}|-(d_X(y)-d_{X_i}(y))\le \frac{n}{4}+3\ga n-\left(\frac{n}{4}-\frac{\beta n}{2}\right)+\ga n\le 5\ga n$$ 
vertices in $X_{3-i}$, as shown in Figure~\ref{fig5a}. Hence, for fixed $X_1$ and $X_2$, the number of ways to choose $N(y)$ for any $y\in Y$ is at most ${|X_1|\choose 5\ga n}{|X_2|\choose 5\ga n}$. Since the number of graphs in $\cH_1$ is precisely the number of ways to add the $[X,Y]$-edges in (P2), we have 
$$|\cH_1|\le 2^{|X|}\cdot 2^{|Y|}\cdot \left({|X_1|\choose 5\ga n}{|X_2|\choose 5\ga n}\right)^{|Y|}\le 2^{\ga^{1/2}n^2},$$
where the first two terms count the number of ways to partition $X=X_1\cup X_2$ and $Y=Y_1\cup Y_2$, and the last term bounds the number of ways to choose the $[X,Y]$-edges.

\begin{figure*}[t!]
    \centering
    \begin{subfigure}[t]{0.4\textwidth}
        \centering
        \includegraphics{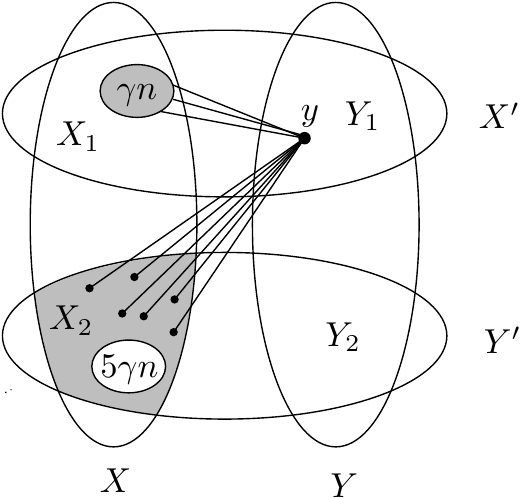}
        \caption{The structure of graphs in $\cH_1$ with max-cut $X'\cup Y'$.}\label{fig5a}
    \end{subfigure}%
    \hspace{2cm}
        \begin{subfigure}[t]{0.4\textwidth}
        \centering
        \includegraphics{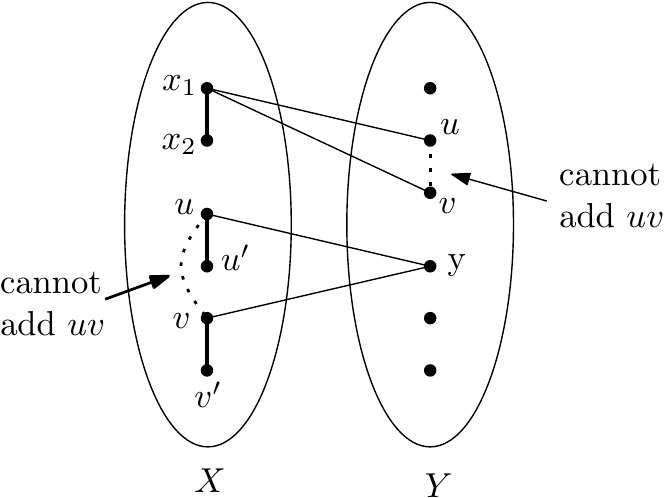}
        \caption{Two examples of $uv$ that cannot be added when forming a graph from $\cH_2$.}\label{fig5b}
    \end{subfigure}
  \caption{}
\end{figure*}

We now bound the number of graphs in $\cH_2$. For any graph $G\in \cH_2$, some inner edges were added in (P3). Suppose that $[X,Y]$-edges added in (P2) were chosen randomly (each of $x_1y$ and $x_2y$ with probability $1/2$). Clearly, $uv$ can be added in (P3) if and only if $u$ and $v$ have no common neighbor. Consider the case when $u,v\in X$ and let $uu', vv'$ be the corresponding edges in $M$ (see Figure~\ref{fig5b}). Every $y\in Y$ is adjacent to exactly one of $u,u'$ and exactly one of $v,v'$. Thus the probability that $y$ is a common neighbor of $u$ and $v$ is $1/4$, which implies that $uv$ can be added with probability $(3/4)^{|Y|}$. Let now $u,v\in Y$. Then $u$ and $v$ have no common neighbor if and only if for every $x_1x_2\in M$, $u$ and $v$ chose different neighbors among $x_1$ and $x_2$. So in this case we can add $u,v$ with probability $(1/2)^{|X|/2}.$ Summing over all possible outcomes of (P2) and all possible choices for $uv$ implies

$$|\cH_2|\le 2^{|X||Y|/2}\cdot {n\choose 2}\cdot \left(\left(\frac{1}{2}\right)^{|X|/2}+ \left(\frac{3}{4}\right)^{|Y|}\right)\ll 2^{|X||Y|/2-n/5}.$$
Hence, we have
$$|\cH'\setm \cG|=|\cH_1|+|\cH_2|\le 2^{\ga^{1/2}n^2}+ 2^{|X||Y|/2-n/5}=o(2^{|X||Y|/2}).$$
\end{proof}

Since $s,t,r\le n$, Lemmas~\ref{claim-st},~\ref{claim-r} and~\ref{claim-g} imply~\eqref{eq-ratio}:

$$\frac{\left|\displaystyle\bigcup_{s,t}\cB(s,t)\setm \cB(0)\right|}{|\cG|}=\frac{\displaystyle\sum_{s,t:~ s+t\ge 1}|\cB(s,t)|+\displaystyle\sum_{r\ge 1}|\cB(r)|}{|\cG|}\le \frac{(n^2+n)\cdot 2^{|X||Y|/2-n/200}}{(1+o(1))2^{|X||Y|/2}}\le 2^{-n/300}.$$

\medskip

\noindent\textbf{Case 2:} $|X|$ is odd.

\medskip

Fix an arbitrary choice of $X,Y,M$ with $|X|$ odd and let $x\in X$ be the vertex not covered by $M$. By Lemmas~\ref{claim-st} and~\ref{claim-r}, 
$$\left|\bigcup_{s,t}\cB(X,Y,M,s,t)\right|\le \displaystyle\sum_{s,t:~ s+t\ge 1}|\cB(X,Y,M,s,t)|+\displaystyle\sum_{r\ge 1}|\cB(X,Y,M,r)|\le 2^{|X||Y|/2-n/300}.$$

Pick an arbitrary vertex $y\in Y$, define $X_0=X\cup\{y\}, Y_0=Y\setm\{y\}$ and $M_0=M\cup\{xy\}$. Then by Lemma~\ref{claim-g}, we have 
$$|\cG(X_0,Y_0,M_0)|\ge (1+o(1))2^{|X_0||Y_0|/2}\ge 2^{|X||Y|/2-(|X|-|Y|)/2-1}\ge 2^{|X||Y|/2-2\beta n},$$
since $|X|-|Y|\le 2\beta n$. Notice that any $(X_0,Y_0,M_0)$ with $|X_0|$ even can be obtained from at most $n$ different triples $(X,Y,M)$ with $|X|$ odd in this way. Together with Lemma~\ref{lem-overcount}, it is sufficient to show that $\bigcup_{s,t}\cB(X,Y,M,s,t)$ is negligible compared to $\cG(X_0,Y_0,M_0)$:
$$\frac{\left|\bigcup_{s,t}\cB(X,Y,M,s,t)\right|}{|\cG(X_0,Y_0,M_0)|}\le\frac{2^{|X||Y|/2-n/300}}{2^{|X||Y|/2-2\beta n}}\le 2^{-n/400}.$$
This completes the proof of Theorem~\ref{thm_main}.


\begin{thebibliography}{99}
\bibitem{abms14}
N.~Alon, J.~Balogh, R.~Morris and W.~Samotij,
\newblock{Counting sum-free sets in Abelian groups},
\newblock{\em Israel Journal of Mathematics}, to appear.

\bibitem{BBS2} J.~Balogh, B.~Bollob\'as and M.~Simonovits, 
\newblock{\em The typical structure of graphs without given excluded subgraphs}, 
\newblock{Random Structures \& Algorithms}, 34, 305--318, 2009.

\bibitem{B6}
J.~Balogh, N.~Bushaw, M.~Collares Neto, H.~Liu, R.~Morris and M.~Sharifzadeh,
\newblock {\em The typical structure of graphs with no large cliques}, 
\newblock submitted.

\bibitem{BB}
  J.~Balogh, and J.~Butterfield, 
  \newblock {\em Excluding induced subgraphs: critical graphs}, 
  \newblock Random Structures \& Algorithms, 38, 100--120, 2011.

\bibitem{Balogh} J.~Balogh, R.~Morris, and W.~Samotij, 
\newblock {\em Independent sets in hypergraphs},  
\newblock to appear in J. Amer. Math. Soc.

\bibitem{BDDLS}
J.~Balogh, S.~Das, M.~Delcourt, H.~Liu, and M.~Sharifzadeh,
\newblock {\em Intersecting families of discrete structures are typically trivial}, 
\newblock to appear in J. Combin. Theory Ser. A.

\bibitem{BLST}
J.~Balogh, H.~Liu, M.~Sharifzadeh and A.~Treglown,
\newblock {\em The number of maximal sum-free subsets of integers}, 
\newblock to appear in Proc. Amer. Math. Soc.

\bibitem{BLST2}
J.~Balogh, H.~Liu, M.~Sharifzadeh and A.~Treglown,
\newblock {\em Sharp bound on the number of maximal sum-free subsets of integers}, 
\newblock in preparation.

\bibitem{Balogh2}
J.~Balogh, R.~Morris, W.~Samotij, and L.~Warnke, 
\newblock {\em The typical structure of sparse $K_{r+1}$-free graphs}, 
\newblock to appear in Trans. Amer. Math. Soc.

\bibitem{BMu1} 
J.~Balogh, and D.~Mubayi, 
\newblock{\em Almost all triple systems with independent neighborhoods are semi-bipartite},  
\newblock{J. Combin. Theory Ser. A}, 118, 1494--1518, 2011.


\bibitem{BP}
J.~Balogh and \v{S}.~Pet\v{r}\'{i}\v{c}kov\'{a}, \newblock {\em Number of maximal triangle-free graphs}, \newblock Bull. London Math. Soc. (2014) 46 (5): 1003--1006.

\bibitem{EFR}

P.~Erd\H{o}s, P.~Frankl, and V.~R\"{o}dl, \newblock {\em The asymptotic number of graphs not containing a fixed
subgraph and a problem for hypergraphs having no exponent}, \newblock Graphs and Combinatorics 2 (1986), 113--121.

\bibitem{Erdos}
P. Erd\H{o}s, D. J. Kleitman and B. L. Rothschild, \newblock {\em Asymptotic enumeration of $K_n$-free graphs}, \newblock Colloquio Internazionale sulle Teorie Combinatorie, Tomo II, Atti dei Convegni Lincei, No. 17, 19--27, Accad. Naz. Lincei, Rome, 1976.


\bibitem{ES} 
 P.~Erd\H{o}s and M.~Simonovits,
 \newblock {\em  A limit theorem in graph theory}, \newblock Studia Sci. Math. Hungar., 1 (1966), 51--57.

\bibitem{Furedi}
Z.~F\"uredi, \newblock {\em The number of maximal independent sets in connected graphs},  \newblock Journal of Graph Theory, Vol.~11, No.~4, 463--470, 1987.

\bibitem{G}
B.~Green,
\newblock A {S}zemer\'edi-type regularity lemma in abelian groups, with
              applications,
\newblock {\em Geom. Funct. Anal.}, 15, (2005), 340--376. 

\bibitem{G-R}
B. Green and I. Ruzsa,
\newblock
Counting sumsets and sum-free sets modulo a prime,
\newblock {\em Studia Sci. Math. Hungarica}, 41, (2004), 285--293. 

\bibitem{HT}
 M.~Hujter and Z.~Tuza,  
\newblock {\em The Number of Maximal Independent Sets in Triangle-Free Graphs}, 
 \newblock SIAM J. Discrete Math., 6(2), 284--288, 1993.
  
\bibitem{KPR} P. G.~Kolaitis, H. J.~Pr\"omel and B. L.~Rothschild, 
\newblock{\em $K_{\ell+1}$-free graphs: asymptotic structure and a 0-1 law}, 
\newblock Trans. Amer. Math. Soc., 303, 637--671, 1987.

\bibitem{MM}
J. W.~Moon and L.~Moser,
\newblock {\em On cliques in graphs},
\newblock  Israel J. Math., 3, (1965), 23--28.

\bibitem{OPT} D.~Osthus, H.~J.~Pr\"omel and A.~Taraz, 
\newblock{\em For which densities are random triangle-free graphs almost surely bipartite?}, 
\newblock Combinatorica, 23, 105--150, 2003.

\bibitem{PSch} 
Y.~Person and M.~Schacht, 
\newblock{\em Almost all hypergraphs without {F}ano planes are bipartite}, 
\newblock Proceedings of the {T}wentieth {A}nnual {ACM}-{SIAM}
              {S}ymposium on {D}iscrete {A}lgorithms, 217--226, 2009.

\bibitem{PS} H.~J.~Pr\"omel and A.~Steger, 
\newblock{\em The asymptotic number of graphs not containing a fixed color-critical subgraph}, 
\newblock Combinatorica, 12, 463--473, 1992.
  
\bibitem{Ruzsa}
I. Z. Ruzsa and E. Szemer\'{e}di, \newblock {\em Triple systems with no six points carrying three triangles}, in
Combinatorics (Keszthely, 1976), Coll. Math. Soc. J. Bolyai 18, Volume II, 939--945.

\bibitem{Saxton}  
D. Saxton and A. Thomason, \newblock {\em Hypergraph containers}, \newblock arXiv:1204.6595.
  
\bibitem{Simonovits}
  M.~Simonovits, \newblock {\em Paul Erd\H{o}s' influence on extremal graph theory}, \newblock The Mathematics of Paul Erd\H{o}s, II, 148-192, Springer-Verlag, Berlin, 1996.
  

\end{thebibliography}
\end{document}